\documentclass[a4paper,11pt]{article}
\usepackage{mathtools}
\usepackage{amsfonts,amssymb,amsmath,amsthm,latexsym}
\usepackage{bbm}
\usepackage{ascmac,fancybox,framed}
\usepackage{enumerate}
\usepackage{listings}
\usepackage{xcolor}
\usepackage{colortbl,color}
\usepackage{float}
\usepackage{graphicx,graphics}
\usepackage{hyperref}
\usepackage[backend=biber, style=numeric, sorting=nyt, giveninits=true]{biblatex}
\allowdisplaybreaks[1]
\mathtoolsset{showonlyrefs}
\newtheorem{Thm}{Theorem}[section]
\newtheorem{Def}[Thm]{Definition}
\newtheorem{Lem}[Thm]{Lemma}
\newtheorem{Prop}[Thm]{Proposition}
\newtheorem{Cor}[Thm]{Corollary}
\newtheorem{Rem}[Thm]{Remark}
\newtheorem{Assump}[Thm]{Assumption}

\begin{document}

\title{An Explicit Representation of the Dominant Eigenstructure for Positive Operators on Banach Lattices} 
\author{Yuki Chino\thanks{Department of Applied Mathematics, 
National Yang Ming Chiao Tung University, Taiwan. 
\url{y.chino@math.nctu.edu.tw}},\\
Kensaku Kinjo\thanks{Academic Support Center, Kogakuin University, Japan.
\url{ft41140@ns.kogakuin.ac.jp}},\\
Ryo Oizumi\thanks{International Relations, National Institute of Population and Social Security Research, Japan. 
\url{ooizumi-ryou@ipss.go.jp}}}
\date{\empty}
\maketitle
\abstract{
The Riesz projection and the corresponding eigenfunction of a positive operator satisfying the Doeblin condition are explicitly constructed using the partial Bell polynomials. While classical Fredholm theory requires stringent summability conditions, such as the operator being in a Schatten class to ensure the convergence of Fredholm minors, our approach utilizes the local algebraic structure induced by the Doeblin condition. We define a scalar function $D(\lambda)$ whose derivative $D'(\lambda_0)$ at the dominant eigenvalue $\lambda_0$ naturally provides the normalization constant for the projection. Consequently, an explicit functional representation of the eigenfunction is obtained as a limit of a weighted ratio of the operator's kernel, bypassing the need to solve transcendental characteristic equations.
}

% Keywords
%\keywords{Doeblin-type minorization, Banach lattice, Bell polynomials, eigenfunction}
%\pacs[MSC Classification]{47B65, 47A10, 47N60, 46B42}

%%%%%%%%%%
\section{Introduction} \label{s:intro}

The theory of positive operators is a field that has attracted the interest of a lot of researchers. In finite dimensions, structures such as eigenvalues and eigenfunctions of the operator can be obtained by solving the characteristic equation; however, due to algebraic constraints, it is not easy to determine their exact values and expressions. It is well known that Perron--Frobenius theorem provides the most familiar example, guaranteeing the existence of simple dominant eigenvalues and positive eigenvectors for irreducible, non-periodic, non-negative matrices. In infinite dimensions, the Krein--Rutman theorem is known as a generalization of Perron--Frobenius theorem; in particular, the eigenvalue and eigenfunction structure of trace class and Hilbert--Schmidt-type operators can be analyzed by considering the determinant defined as a series. However, these classical approaches heavily rely on the Schatten class properties of the operator, which require the decay of singular values to be sufficiently fast for the Fredholm determinant to be well-defined. In many practical applications, such as those involving general Markov operators or non-smooth kernels, these summability conditions are often too restrictive or difficult to verify. Furthermore, even when the determinant exists, relating its zeros back to the explicit construction of spectral projections remains a non-trivial challenge. In particular, when the decay of eigenvalues is slow (e.g. as seen in certain logarithmic kernels), classical theory necessitates sophisticated regularization procedures such as Carleman--Fredholm subdeterminants, which often obsecure the direct algebraic connection between the operator iterates and its eigenfunctions.

To be more precise, for trace class integral kernel $K$, we consider the following positive operator:
\begin{equation}
    (T f)(x) = \int K(x,y) \, f(y) \; dy.
\end{equation}
Then the resolvent $(\lambda I - T)^{-1}$ yields
\begin{equation}
    [(\lambda I - T)^{-1} g](x) = \frac{1}{\lambda} g(x) + \frac{1}{\lambda^2} \int R(x,y;\lambda) \, g(y) \; dy
\end{equation}
and
\begin{equation}
    R(x,y;\lambda) = \frac{D_F(x,y;\lambda)}{D_F(\lambda)}.
\end{equation}
In the above expression $D_F(\lambda) = \det (I - \frac{1}{\lambda}T)$ is the Fredholm determinant and the Fredholm minor $D_F(x,y;\lambda)$ is expressed as a specific series. The eigenvalue is characterized by a $\lambda_0$ satisfying $D_F(\lambda_0) = 0$, which exactly corresponds to solving the characteristic equation in finite dimensions. At simple eigenvalue $\lambda_0$, the Fredholm minor can be written by eigenfunction and its adjoint: $D_F(x,y;\lambda_0) = C \phi(x) \psi(y)$ with some constant. Thus the eigenfunction $\phi(x)$ has the explicit expression $D_F(x,y_0;\lambda_0)$, where $y_0$ such that $\psi(y_0) \neq 0$. However, the convergence of the specific series determining the Fredholm minor and determinant highly relies on the decay speed of singular values. The trace class assumption guarantees the absolute convergence of the series. 

In this paper, we overcome these limitations by shifting the focus from the global summability of singular values to the local algebraic structure induced by the Doeblin condition. Instead of treating the spectral problem as a static transcendental equation $D_F(\lambda) = 0$, we propose a constructive and dynamic framework. By employing the partial Bell polynomials, we construct a scalar function $D(\lambda)$ that characterizes the spectral gap and the Riesz projection without requiring the operator to belong to any specific Schatten class. This approach provides not only the existence of the dominant eigenvalue but also an explicit, computable representation of the corresponding eigenfunction as a limit of algebraic combinations of the operator kernel.

Our first contribution on this paper is a determinant-free spectral characterization that bypasses the structural constraints. Without assuming compactness, trace class properties, or nuclearity, we derive a factorization of the resolvent via a rank-one inversion formula. Isolated eigenvalues are characterized as zeros of a scalar analytic function--one might identify as the generalized Birman--Schwinger-type function in quantum theory even though it is defined under Hilbert--Schmidt--and the associated spectral projections appear as residues. This scalar reduction remains valid even in the absence of compact embeddings, offering a significant generalization of traditional Fredholm-type results. In the theory of stability for nonlinear traveling waves or fluid, we also sometimes consider the Evans function to characterize the eigenvalue as a zero of $D(\lambda)$.

We also present a novel, resolvent-centric framework to characterize the dominant spectral behavior of positive operators subject to a rank-one Doeblin-type condition. Positive operators on Banach lattices play a fundamental role in spectral theory. Classical results, such as the Perron--Frobenius theorem and the Krein--Rutman theorem, establish the existence and the dominance of a positive eigenvalue under strong positivity or compactness assumptions \cite{P1907,F1912,KR1948,Lotz1985}. A common structural hypothesis underlying these results is a rank-one lower bound, or Doeblin-type minorization, which ensures positivity improvement and irreducibility.

Another contribution is a kernel-level resolvent factorization, which explicitly captures how the rank-one contribution propagates through the operator iterates. By recursively subtracting the rank-one component, we introduce corrected kernels that yield a Neumann-type expansion of the resolvent with explicit rank-one corrections. This representation allows both resolvent kernels and eigenfunctions to be written explicitly in terms of the original integral kernel. By identifying the spectral projection as a rank-one residue, we provide an explicit algebraic template for asymptotic analysis, applicable to a wide range of integral projection models (IPMs).

%%%%%
\subsection{Structural foundations}

%%%%%
\paragraph{Banach lattice and order continuity}
Let $(X,\mathcal{F},\mu)$ be a $\sigma$-finite measure space, and let $E$ be a Banach function space over $(X,\mathcal{F},\mu)$ satisfying Fatou property and order-continuous norm. We denote by $E'$ the Banach dual of $E$ and by $E^\times$ the associate space (K\"othe dual) of $E$, defined by
\begin{equation}
    E^\times := \Bigl\{ g \colon \int_X |f \, g| \; d\mu < \infty \; \text{ for all } f \in E \Bigr\}, \qquad \|g\|_{E^\times} := \sup_{\|f\|_E \leq 1} \int_X |f g| \, d\mu.
\end{equation}
Each $g \in E^\times$ induces a continuous linear functional on $E$ via
\begin{equation}
    \Phi_g(f) := \int_X f(x) \, g(x) \; d\mu(x), \qquad f \in E.
\end{equation}
Since we assume that $E$ has the Fatou property and an order-continuous norm, standard results in the theory of Banach function spaces imply that every continuous linear functional on $E$ arises in this way, and the map $g \mapsto \Phi_g$ yields an isometric isomorphism
\begin{equation}
    E' \cong E^\times.
\end{equation}
The assumptions: Fatou property and order-continuous norm ensure the identification $E' \simeq E^\times$ and exclude pathologies such as the failure of order continuity in $L^\infty$. We refer to \cite{BS1988} and \cite{MN1991} for the duality theory of Banach function spaces.

Restricting to the positive cone, we obtain
\begin{equation}
    E'_+ = \left\{ \Phi_g \colon \Phi_g(f) = \int_X f(x) g(x) \; d\mu(x), \quad g \in E^\times_+ \right\}.
\end{equation}
and hence any positive functional $\Phi \in E'_+$ can be written as
\begin{equation}
    \Phi[f] = \int_X f(x) g(x) \; d\mu(x), \qquad g \in E^\times_+
\end{equation}
The identification $E' \cong E^\times$ reduces the abstract functional $\Phi$ to an integral with the concrete weight function $g(y)$. This is the key that elevates the characteristic function $D(\lambda)$ from mere \emph{existence} to a \emph{concrete computational object}.

We consider a measurable kernel $K \colon X \times X \to [0,\infty)$ inducing the positive operator
\begin{equation}
    (Tf)(x) = \int_X K(x,y) \, f(y) \; d\mu(y) \qquad f \in E.
\end{equation}
We assume throughout that $T$ is well-defined and bounded on $E$.

%%%%%
\paragraph{Doeblin-type minorization}
The central structural pillar of our analysis is the following Doeblin-type rank-one minorization condition.

\begin{Assump}[Doeblin-type minorization] \label{assump:doeblin}
    \rm There exist $\alpha > 0$, a function $u_0 \in E_+$, and a nonnegative measurable function $g$ such that
    \begin{equation} \label{eq:doeblin_kernel}
        K(x,y) \geq \alpha u_0(x) g(y) \qquad \forall x,y \in X.
    \end{equation}
    We denote by $\Phi \in E'_+$ a strictly positive functional, i.e.
    \begin{equation}
        \Phi[f] > 0 \qquad \forall f\in E_+\setminus\{0\},
    \end{equation}
    where $g(y) > 0$ $\mu$-a.e. ensures that $\Phi$ is strictly positive.
\end{Assump}

We associate to $g \in E^\times$ and the functional
\begin{equation}
   \Phi[f] := \int_X f(y) g(y) \; d \mu(y). 
\end{equation}
The minorization condition is classical in Markov and Perron--Frobenius theory and provides the structural foundation for the analysis that follows.

%%%%%
\subsection{Main result}
We first show, under a rank-one Doeblin minorization, that the operator $T$ can be separated into rank-one and the remainder.

\begin{Prop} \label{prop:quasi-compact}
    Under Assumption~\ref{assump:doeblin}, the operator admits the decomposition
    \begin{equation}
        T = \alpha P_0 + R,
    \end{equation}
    where $P_0$ is rank-one and $R \geq 0$.
\end{Prop}
\noindent This decomposition effectively identifies $T$ as a quasi-compact operator on the Banach lattice. The rank-one operator $P_0$ represents the dominant channel of the dynamics, while $R$ encapsulates the essential spectrum and transient components. Unlike the classical perturbation theory of Kato~\cite{K1966}, our method does not treat $R$ as asmall perturbation but rather as a structural remainder that can be handled through the inversion formula without global trace-class assumptions.

Next we show, under a rank-one Doeblin minorization, that the isolated spectrum of a positive operator $T$ is completely characterized by the zeros of an explicit scalar analytic function $D(\lambda)$, which is a direct consequence of Sherman--Morrison formula. 

\begin{Prop} \label{prop:resolvent_decomposition}
    Let $\lambda \in \rho(R)$ and assume
    \begin{equation}
        1 - \alpha \, \Phi[R_\lambda u_0] \neq 0
    \end{equation} 
    Then the resolvent of $T$ is given by the explicit formula
    \begin{equation} \label{eq:sm-resolvent}
        (\lambda I - T)^{-1} \, = \, R_\lambda + \frac{\alpha \, (R_\lambda u_0) \otimes (\Phi R_\lambda)}{1 - \alpha \, \Phi[R_\lambda u_0]}.
    \end{equation}
\end{Prop}

\noindent In the end, we obtain the explicit representation of eigenfunction by analyzing the function $D(\lambda) = 1 - \alpha \, \Phi[R_\lambda u_0]$.

\begin{Thm} \label{thm:eigenfunction}
    The eigenfunction of $T$ is
    \begin{equation}
        w = \frac{\alpha}{D'(\lambda_0)} \sum_{n=0}^\infty \lambda_0^{-n} \, \Gamma_n,
    \end{equation}
    where $\lambda_0$ is the zero of $D(\lambda)$ and $\Gamma_{n+1} = R \, \Gamma_n$ with $\Gamma_0 = u_0$. 
\end{Thm}

\noindent While Birman--Schwinger-type characterizations are known in existence proofs of $\lambda_0$ \cite{H2007,GM2025}, to the best of our knowledge, the algebraic simplicity of the dominant pole in non-compact Banach lattice settings has not been established via this analytic zero structure without assuming the existence of the determinant.

\smallskip

Unlike the classical Fredholm theory, our approach provides a constructive bridge between the iteration of the operator and its spectral components. By identifying the expansion coefficients with the partial Bell polynomials, we encapsulate the complex recurrence of the resolvent into a tractable algebraic framework.

\begin{Thm}
    Let $T$ be a bounded positive operator possible to decompose into rank-one $P$ and the remainder $R$. For $\Gamma_n = R \, \Gamma_{n-1} = (T - P) \Gamma_n$, we have the following expression 
    \begin{equation}
        \Gamma_n = K^{(n)} - \sum_{\ell=0}^{n-1} (-1)^\ell \sum_{k=0}^{n-\ell-1} K^{(n-k-\ell-1)} \; B_{\ell+1,k+\ell+1}(b_1,b_2,\cdots),
    \end{equation}
    where $b_j := \Phi_\xi(K^{(j)}u_0)$ and $B_{p,q}(b_1,b_2,\cdots)$ are the partial Bell polynomials with parameter $p,q$.
\end{Thm}

\noindent This theorem leads to one of our most significant contributions: an explicit functional representation of the eigenfunction, which bypasses the need to solve transcendental characteristic equations and offers a direct pathway from the operator kernel to its steady-state structure. This representation uncovers the genealogical feedback mechanism where the history of observations $b_n$ is systematically reorganized via partial Bell polynomials to reconstruct the steady-state profile.

%%%%%
\subsection{Discussion}
In the classical Fredholm theory, Hilbert--Schmidt space or trace class properties has been assumed so that Fredholm determinant is well-defined. While the present analysis is formulated on Banach lattices with order-continuous norms, the resolvent-based structure suggests that parts of the theory extend to more general $L^p$-$L^q$ frameworks ($p\in[1,\infty)$) and different underlying measures. The functional-analytic framework adopted in this paper is not tied to a specific choice of function space. As shown in Appendix~A, the resolvent factorization and the associated rank-one structure are invariant under natural changes of measure and remain stable across a broad class of Banach function spaces with order-continuous norms. Beyond this level of robustness, further extensions to more general $L^p$-$L^q$ settings or to spaces without order-continuous norms would require a reformulation of the duality and positivity arguments.

At a most concrete level, one may consider kernel-level recursions based on point-evaluation subtraction, which we will consider in Appendix~C, such as
\begin{equation}
    \begin{aligned}
        &\Gamma_{0}(x,y) = K(x,y),\\
        &\Gamma_{n+1}(x,y) = \int K(x,\xi) \Gamma_n(\xi,y) \, d\mu(\xi) - K(x,y) \, \Gamma_n(x_0,y_0), \quad n \geq 0.
    \end{aligned}
\end{equation}
This recursion corresponds to point-evaluation functionals, which naturally arise in $L^\infty$ settings but fall outside the class of spaces with order-continuity. While mollification procedures may be used to approximate point-evaluations, establishing convergence of resulting recursions requires additional regularity assumptions. This suggests potential extensions to more general setup, such as Orlicz space which generalizes $L^p$ spaces. In general, the norm is not order-continuous in Orlicz space. However, Doeblin condition and kernel lift with mollification may give a crucial idea for the asymptotic behavior on a space where functions have unbounded growth.

%%%%%
\subsection{Organization of the paper}
In Section~2, we prove Proposition~\ref{prop:quasi-compact} establishing positivity improving and irreducibility under a Doeblin-type condition.
In Section~3, we develops the rank-one inversion known as Sherman--Morrison formula and derives the decomposition of the resolvent, which is stated in Proposition~\ref{prop:resolvent_decomposition}. We also give the proof of Theorem~\ref{thm:eigenfunction}.
Section~4 applies the facts to the spectral theory of $T$, establishing the existence, uniqueness, and simplicity of the dominant eigenvalue and its characterization via $D(\lambda)$.
Section~5 constructs the corrected kernels and shows the kernel-level representation.
Appendix~\ref{s:change_meas} discusses invariance under changes of measure, and Appendix~\ref{s:PF} presents Perron--Frobenius-type results and examples illustrating power-Doeblin conditions as stated in \cite{A2025, H2007}. Appendix~\ref{s:kernel-rankone} yields that a constructive kernel recursion provides stable computation for integral projection models. In the end, Appendix~\ref{s:numerical} yields an example and numerical results showing an advantage of our approach comparing to the existing theory.

%%%%%%%%%%
\section{Structural Properties of the Operator} \label{s:structure}

%%%%%
\subsection{Positivity-improving property and irreducibility}
Minorization immediately yields a lower bound for the action of $T$ on positive functions:
\begin{equation} \label{eq:operator_T}
    T f (x) = \int_X K(x,y) \, f(y) \; d\mu(y) \geq \alpha u_0(x) \Phi[f] \qquad \forall f \in E_+.
\end{equation}

\begin{Lem} \label{lem:positive-improving}
    Assume Assumption\ref{assump:doeblin} and that $\Phi$ is strictly positive. Then
    \begin{itemize}
        \item[\bf 1. ] $T$ is positive improving, i.e.
        \begin{equation}
            f \geq 0, \, f \not\equiv 0 \quad \Rightarrow \quad T f > 0.
        \end{equation}
        \item[\bf 2.] $T$ has no nontrivial closed $T$-invariant band of $E$.
        \item[\bf 3.] Consequently, $T$ is irreducible.
    \end{itemize}
\end{Lem}

\begin{proof}
    Recall that a band in E is a solid, order-closed ideal. Positivity improving follows immediately from \eqref{eq:operator_T}: for any $f \in E_+\setminus \{0\}$,
    \begin{equation}
        T f(x) \geq \alpha u_0(x) \Phi[f] > 0.
    \end{equation}
    Let $B \subset E$ be a closed $T$-invariant band. Choose $0 \neq f \in B \cap E_+$. Then $T f \in B$, and thus \eqref{eq:operator_T} implies $u_0 \in B$. For any $g \in E_+\setminus \{0\}$,
    \begin{equation}
        T g \geq \alpha u_0 \Phi[g] > 0,
    \end{equation}
    so $g \in B$, hence $B = E$. Thus no nontrivial closed $T$-invariant band exists and $T$ is irreducible.
\end{proof}

%%%%%
\subsection{Proof of Proposition~\ref{prop:quasi-compact}}
Define the rank-one operator
\begin{equation}
    P_0 = u_0 \otimes \Phi.
\end{equation}
The minorization inequality \eqref{eq:operator_T} can be rewritten as an operator inequality:
\begin{equation} \label{eq:doeblin_operator}
    T f \geq \alpha P_0 f \qquad \forall f \in E_+.
\end{equation}
From Lemma~\ref{lem:positive-improving}, we take the following positive operator
\begin{equation}
R := T - \alpha P_0.
\end{equation}
Then $R > 0$ and
\begin{equation} \label{eq:decomposition_T}
    T = \alpha P_0 + R.
\end{equation}
Since $P_0$ is rank-one, it is compact and hence $T$ is a compact perturbation of the positive operator $R$. Also since $P_0$ is compact, the essential spectral radius of $T$ coincides with that of $R$. The Doeblin condition serves to push down this radius of $R$ below the dominant eigenvalue of $T$.

Under the Doeblin-type rank-one minorization such an operator is quasi-compact; see, e.g., \cite{S1974, S2006} or \cite{H2007}. The next section develops the resolvent identity and the Birman--Schwinger resolvent factorization arising from the rank-one nature of $P_0$. Related quasi-compactness results for positive operators under minorization or uniform positivity conditions can be found in \cite{S1974} and \cite{MT2009}.

\begin{Rem}[Boundedness of operators]
    \rm The boundedness of $T$ yields the Neumann series expression for $|\lambda| > \|T\|$, which gives the resolvent positivity. The boundedness also yields the positivity of $e^{s T}$ for any $s \geq 0$. Thus
    \begin{equation}
        (\lambda I - T)^{-1} = \int_0^\infty e^{-\lambda s} \, e^{s T} \; ds \geq 0,
    \end{equation}
    which shows resolvent positivity of $T$. On the other hand, the boundedness of $R$ is guaranteed by the decomposition \eqref{eq:decomposition_T} and the boundedness of $T$. This yields the resolvent positivity of $R$ for $\lambda > \rho(R)$. Then, by simple computation with the resolvent identity, the resolvent of $R$ is decreasing in $\lambda$, which will be crucial for the uniqueness of the dominant eigenvalue.
\end{Rem}

%%%%%%%%%%
\section{Construction of the Characteristic Equation} \label{s:characteristic}

In this section we will show the explicit form of an analytic function that yields spectral projection onto the eigenspace of the operator $T$ as a residue at some $\lambda_0$. Though the philosophy of analysis follows the traditional manner to some extent, we consider more general setting and the analytic function is no longer defined as the determinant, such as Evans function in the theory on nonlinear wave or fluid equations. As in the previous section, we focus on the operator
\begin{equation}
    T = \alpha P_0 + R, \qquad P_0 = u_0 \otimes \Phi,
\end{equation}
where $P_0$ is a rank-one and $R \geq 0$. This decomposition enables an exact formula for the resolvent $(\lambda I - T)^{-1}$ separating regular and singular parts.

%%%%%
\subsection{Inversion formula for rank-one perturbations}
We begin with an inversion formula for rank-one operators. Let $a \in E_+$ and $b \in E'_+$. Consider the rank-one operator
\begin{equation}
    S = a \otimes b.
\end{equation}
We record the classical inversion identity, stated here for completeness. The following inversion formula for rank-one operators as known as Sherman--Morrison formula, is classical (see \cite{B1951, K1966}).

\begin{Lem}[Sherman--Morrison formula \cite{SM1950, B1951}] \label{lem:rank-one}
Let $a \in E_+$ and $b \in E'_+$. If $1 - b[a] \neq 0$, then the operator $I - a \otimes b$ is invertible and
\begin{equation}
    (I - a \otimes b)^{-1} = I + \frac{a \otimes b}{1 - b[a]} \qquad b[a] \neq 1
\end{equation}
\end{Lem}

\begin{proof}
    Multiply $(I - a \otimes b)$ for both sides. The right-hand side is
    \begin{equation}
        \left( I + \frac{a \otimes b}{1 - b[a]} \right) (I - a \otimes b) = I - a \otimes b + \frac{a \otimes b}{1 - b[a]} - \frac{(a \otimes b)(a \otimes b)}{1 - b[a]}
    \end{equation}
    Since for any $x \in E$
    \begin{equation}
        (a \otimes b)(a \otimes b)(x) = (a \otimes b)(a \, b[x]) = a \, b[a] \, b[x] = b[a] \, a \, b[x] = b[a] \, (a \otimes b)(x),
    \end{equation}
    the multiplication of $(a \otimes b)$ is the same as multiplication of $b[a]$. Then the last two terms can be computed as
    \begin{equation}
        \frac{a \otimes b}{1 - b[a]} - \frac{(a \otimes b)(a \otimes b)}{1 - b[a]} = a \otimes b
    \end{equation}
    which yields $\left( I + \frac{a \otimes b}{1 - b[a]} \right) (I - a \otimes b) = I$. Hence the proof of the lemma is completed.
\end{proof}

\begin{Cor}[Pole and residue of the resolvent of a rank-one operator] \label{cor:pole_residue}
    Let $S := a \otimes b$. Then the resolvent $(\lambda I - S)^{-1}$ admits the representation
    \begin{equation}
        (\lambda I - S)^{-1} = \frac{1}{\lambda} \, I + \frac{a \otimes b}{\lambda (\lambda - b[a])} \qquad \lambda \neq 0, b[a]
    \end{equation}
    At $\lambda = b[a]$, the resolvent has a simple pole with residue
    \begin{equation}
        \text{Res}_{\lambda = b[a]} (\lambda I - S)^{-1} = \frac{a \otimes b}{b[a]}.
    \end{equation}
    If $b[a] = 1$, the residue equals the rank-one projection $a \otimes b$.
\end{Cor}

%%%%%
\subsection{Proof of Proposition~\ref{prop:resolvent_decomposition}}
Let $\lambda \in \rho(T)$, so that the resolvent
\begin{equation}
    R_\lambda = (\lambda I - R)^{-1}
\end{equation}
is well-defined. We factorize
\begin{equation}
    \lambda I - T = (\lambda I - R) - \alpha P_0 = (\lambda I -R) (I - \alpha R_\lambda P_0).
\end{equation}
Since $P_0 = u_0 \otimes \Phi$, the second factor has the form covered by Lemma~\ref{lem:rank-one}:
\begin{equation}
    I - \alpha R_\lambda P_0 = I - \alpha (R_\lambda u_0) \otimes \Phi.
\end{equation}
Thus its inverse exists whenever
\begin{equation}
    1 - \alpha \Phi[R_\lambda u_0] \neq 0.
\end{equation}
This yields the decomposition of the resolvent $(\lambda I - T)^{-1}$.
\begin{equation}
    (\lambda I - T)^{-1} \, = \, R_\lambda + \frac{\alpha \, (R_\lambda u_0) \otimes (\Phi R_\lambda)}{1 - \alpha \, \Phi[R_\lambda u_0]}.
\end{equation}
For $|\lambda|$ is sufficiently large, the right-hand side converges in operator norm and equals the analytic resolvent $(\lambda I - T)^{-1}$.

The analytic denominator in \eqref{eq:sm-resolvent}
\begin{equation}
    D(\lambda) = 1 - \alpha \Phi[R_\lambda u_0] \qquad (\lambda \in \rho(R))
\end{equation}
encodes all isolated eigenvalues of $T$ lying in the resolvent set of $R$.

\begin{Cor} \label{cor:R_1}
    Let $\lambda_0 \in \rho(R)$ satisfy
    \begin{equation}
        D(\lambda_0) = 0, \qquad D'(\lambda_0) \neq 0
    \end{equation}
    Then $(\lambda I - T)^{-1}$ has a simple pole at $\lambda = \lambda_0$ with residue
    \begin{equation}
        \text{Res}_{\lambda = \lambda_0} (\lambda I - T)^{-1} \; = \; \frac{\alpha \, (R_{\lambda_0} u_0)\otimes(\Phi R_{\lambda_0})}{D'(\lambda_0)}.
    \end{equation}
    This operator is rank-one and equals the spectral projection onto the eigenspace of $T$ at $\lambda_0$.
\end{Cor}

%%%%%
\subsection{Renewal structure: Proof of Theorem~\ref{thm:eigenfunction}}
The rank-one perturbation yields a renewal structure. We define the Doeblin-type recursion by
\begin{equation}
\Gamma_0 = u_0, \qquad \Gamma_n = R \, \Gamma_{n-1} \; (n\geq1)
\end{equation}
Then, we have $\Gamma_n = R^n \, u_0$. By Corollary~\ref{cor:R_1}, we express the residue at $\lambda_0$, denoting by $w$,
\begin{equation} \label{eq:eigen_func}
    \begin{aligned}
        w = \lim_{\lambda\to\lambda_0} (\lambda - \lambda_0) \, \frac{\alpha (\lambda I - R)^{-1} u_0 }{D(\lambda)} = \frac{\alpha}{D'(\lambda_0)} (\lambda_0 I - R)^{-1} u_0\\
        %&= \frac{\alpha}{D'(\lambda_0)} \sum_{n=0}^\infty \lambda_0^{-(n+1)} \, \Gamma_n
    \end{aligned}
\end{equation}
%where the Neumann series is well-defined because $\lambda_0 > \rho(R)$. 
The expression \eqref{eq:eigen_func} yields
\begin{equation}
    (\lambda_0 I - R) \, w = \frac{\alpha}{D'(\lambda_0)} u_0.
\end{equation}
Since $D(\lambda_0) = 0$ implies that $\Phi[w] \, D'(\lambda_0) = 1$,
\begin{equation}
    \begin{aligned}
    (\lambda_0 I - T) w &= (\lambda_0 I  - \alpha u_0 \otimes \Phi - R) \, w\\
    &= (\lambda_0 I - R) \, w - \alpha u_0 \Phi[w] = \alpha u_0 \left( \frac{1}{D'(\lambda_0)} - \Phi[w] \right) = 0
    \end{aligned}
\end{equation}
which yields $T w = \lambda_0 w$, that is, $w$ is an eigenvector of $T$.

\medskip

\noindent The rank-one subtraction $T = P + R$ makes the spectral radius $\rho(R)$ smaller than that of $T$. This yields the well-definedness of several series in this paper.

\begin{Lem} \label{lem:spectral_radius}
    $\rho(R) < \rho(T)$.
\end{Lem}

\begin{proof}
    We first recall that $D(\lambda) = 1 - \alpha \Phi[R_\lambda u_0]$. For $\lambda > \rho(R)$, we have
    \begin{equation}
        \frac{d}{d\lambda} R_\lambda = \frac{d}{d\lambda} (\lambda I - R)^{-1} = -(\lambda I - R)^{-2}
    \end{equation}
    Thus
    \begin{equation}
        \frac{d}{d\lambda} D(\lambda) = - \alpha \Phi_\xi \left[ \frac{d}{d\lambda} R_\lambda u_0 \right] = \alpha \Phi_\xi[R_\lambda^2 u_0] > 0
    \end{equation}
    Therefore, $D(\lambda)$ is increasing in $\lambda > \rho(R)$. Next let $g(\lambda) := \alpha \Phi_\xi[R_\lambda u_0] \geq 0$. Then it is decreasing in $\lambda > \rho(R)$ and $g(\lambda) \to 0$ as $\lambda \to \infty$. Since $R_\lambda$ diverges as $\lambda \downarrow \rho(R)$, $g(\lambda) \to \infty$ as $\lambda \downarrow \rho(R)$. Therefore, by the intermediate theorem, there exists a unique $\lambda_\ast \in (\rho(R),\infty)$ such that $g(\lambda) = 1$. From the fact that $D(\lambda_0) = 0$, i.e. $g(\lambda_0) = 1$ and the classical Krein--Rutman theorem (see Lemma~\ref{lem:existence} in the next section), we have
    \begin{equation}
        \lambda_\ast = \lambda_0 = \rho(T),
    \end{equation}
    which is the unique dominant eigenvalue of $T$. Hence we conclude $\rho(R) < \rho(T)$.
\end{proof}

\smallskip

Hence, by Lemma~\ref{lem:spectral_radius}, we can expand $(\lambda_0 I - R)^{-1}$ in \eqref{eq:eigen_func} and obtain the representation
\begin{equation}
    w = \frac{\alpha}{D'(\lambda_0)} \sum_{n=0}^\infty \lambda_0^{-n} \, \Gamma_n,
\end{equation}
which completes the proof of Theorem~\ref{thm:eigenfunction}.

\medskip

\begin{Rem}
    \rm The characteristic equation yields the condition for the nontrivial solution: $1 = \alpha \Phi[(\lambda I - R)^{-1} u_0]$, which is exactly $D(\lambda) = 0$. Thus the analytic function $D(\lambda)$ describes an alternative representation for eigenstructure on more general functional spaces with weaker assumptions.
\end{Rem}

\begin{Rem}
    \rm When a positive operator $T$ is Riesz decomposable with the simple dominant eigenvalue $\lambda_0$, we have the following 
    \begin{equation}
        T = \lambda_0 P_0 + R, \qquad P_0 = \frac{w_0 \otimes v_0}{\langle v_0, w_0 \rangle},
    \end{equation}
    where $w_0$ and $v_0$ are right and left eigenfunctions, respectively. In our context,
    \begin{equation}
        w_0 = \frac{\alpha R_{\lambda_0} u_0}{D'(\lambda_0)}, \qquad v_0 = \frac{\Phi R_{\lambda_0}}{D'(\lambda_0)}, \qquad \langle v_0, w_0 \rangle = D'(\lambda_0).
    \end{equation}
    Moreover, if $T$ is Markov, i.e., $\lambda_0 = 1$, then $w_0$ and $v_0$ yields the invariant density and distribution, respectively. The asymptotic behavior of Markov process can be given by the one-step information $R_1$.
\end{Rem}

%%%%%%%%%%
\section{Spectral Dominance of the Positive Operator \texorpdfstring{$T$}{}} \label{s:spectral_dominance}

In this section we complete the spectral analysis of the operator
\begin{equation}
    T =\alpha P_0 + R
\end{equation}
constructed in Section~\ref{s:structure}--\ref{s:characteristic}. 
The purpose of this section is to characterize the dominant spectral value of $T$ under the Doeblin condition, which implies primitivity. The absence of other peripheral eigenvalues in this setting should be viewed as a structural consequence of strong positivity, included here to clarify the relation with classical Perron--Frobenius theory. Using quasi-compactness, irreducibility, and the Birman--Schwinger resolvent factorization, we show that $T$ admits a unique strictly dominant eigenvalue, that this eigenvalue equals the spectral radius, and that its eigenspace is one-dimensional. The argument synthesizes the Krein--Rutman theorem with the analytic structure obtained in Section~\ref{s:characteristic}.

%Recall from Section~\ref{s:structure} that $T$ is positive, irreducible, and quasi-compact. Thus the Krein--Rutman theorem applies to $T$ as an operator on the Banach lattice $E$.

%%%%%
\subsection{Existence of a strictly positive eigenvalue}

The following lemma holds directly from the Krein--Rutman theorem. The Krein--Rutman theorem and its generalizations provide a cornerstone for the spectral analysis of positive operators on Banach lattices; see, e.g. \cite{KR1948} for classical formulations, and also \cite{FSGM2023} for recent developments in a semigroup framework covering numerous integro-differential and kinetic examples.

\begin{Lem}[Existence of a positive eigenvalue] \label{lem:existence}
    Since $T$ is positive and quasi-compact, its spectral radius $\rho(T)$ satisfies $\rho(T) > 0$, and there exists a vector $w \in E_+\setminus\{0\}$ such that
    \begin{equation} \label{eq:p-eigenvector}
        T w = \rho(T) \, w.
    \end{equation}
\end{Lem}

\noindent Thus the spectral radius is not only part of the spectrum but is a positive real eigenvalue.

%%%%%%
\subsection{Simplicity of the eigenvalue under irreducibility}

We next show that the eigenspace associated with $\rho(T)$ is one-dimensional, and that $T$ has no other eigenvalue on the spectral circle. This is the direct consequence from primitivity, i.e. the Doeblin condition. Let $v \in E$ satisfy$|\lambda| = \rho(T)$.
\begin{equation}
    T v = \lambda v, \qquad |\lambda| = \rho(T).
\end{equation}
Taking modulus and using positivity of $T$,
\begin{equation}
    \rho(T) |v| = |\lambda v| = |T v| \leq T |v|
\end{equation}
Applying the eigenfunctional $\Phi \in E'_+$ from Section~\ref{s:characteristic} gives
\begin{equation}
    \Phi(T |v|) = \rho(T) \, \Phi(|v|),
\end{equation}
and hence,
\begin{equation}
    T |v| = \rho(T) \, |v|
\end{equation}
Since $T$ is irreducible and positive-improving, the eigenspace associated with a positive eigenvector is one-dimensional. Thus $|v| = c w$ for some $c > 0$, where $w$ is the eigenvector in \eqref{eq:p-eigenvector}. Writing $v = c e^{i\theta} w$ and substituting into $T v = \lambda v$, we obtain
\begin{equation}
    \rho(T) \, w = e^{-i\theta} \lambda w
\end{equation}
Hence $\lambda = \rho(T) e^{i\theta}$. Since $T$ is real and positive, the left-hand side is a positive multiple of $w$, so $\lambda$ is also real and positive. Thus we have the following proposition:

\begin{Lem} \label{lem:simple}
    The eigenvalue $\rho(T)$ is simple, and it is the only eigenvalue of $T$ with modulus $\rho(T)$.
\end{Lem}

\noindent We emphasize that the uniqueness of the peripheral eigenvalue relies on the primitive (one-step Doeblin) condition. For power-Doeblin and related weakenings, the dominant spectral structure persists while the peripheral spectrum may be nontrivial; see Appendix~B.

%%%%%
\subsection{Relation with the resolvent decomposition}

Section~\ref{s:characteristic} shows that the resolvent of $T$ can be written as
\begin{equation} \label{eq:resolvent}
    (\lambda I - T)^{-1} = R_\lambda + \frac{\alpha (R_\lambda u_0) \otimes (\Phi R_\lambda)}{1 - \alpha \Phi[R_\lambda u_0]},
\end{equation}
where
\begin{equation} \label{eq:D}
    D(\lambda) := 1 - \alpha \Phi[R_\lambda u_0].
\end{equation}
A pole of the resolvent can occur only where $D(\lambda) = 0$. Such a pole corresponds to an isolated eigenvalue of $T$, and the residue gives the spectral projection. Since it is shown above that a unique eigenvalue lies on the spectral circle and equals $\rho(T)$, it follows

\begin{Prop} \label{prop:uniqueness}
    The dominant eigenvalue $\rho(T)$ is characterized uniquely by
    \begin{equation}
        D(\rho(T)) = 0.
    \end{equation}
    At this value, the resolvent has a simple pole, and the residue equals the rank-one projection onto the eigenspace spanned by $w$:
    \begin{equation} \label{eq:residue}
        \text{Res}_{\lambda=\rho(T)} ( \lambda I - T)^{-1} = w \otimes \Phi,
    \end{equation}
    with normalization $\Phi[w] = 1$.
\end{Prop}

\noindent Combining Lemmas~\ref{lem:existence} and~\ref{lem:simple} with Proposition~\ref{prop:uniqueness} yields a complete characterization of the dominant spectral value of $T$ under the Doeblin condition. In particular, all Perron–-Frobenius-type conclusions used in this paper are consequences of the Doeblin condition and serve only as input for the characterization in Section~\ref{s:characteristic}.

%%%%%%%%%%
\section{Kernel-Level Resolvent Decomposition} \label{s:BS_expansion}

The decomposition in Section~\ref{s:characteristic} expresses the resolvent of $T$ explicitly at the operator level. In this section we develop a kernel-level version of the decomposition, based on recursively defined corrected kernels. This yields a Neumann-type series for the resolvent $(\lambda I - T)^{-1}$ but with precise corrections accounting for the rank-one perturbation.

We define $\Phi_\xi \in E'_+$ for $g_0 \in E^\times$ by
\begin{equation}
    \Phi_\xi(f(\xi,y)) := \int_{\Omega^d} g_0(\xi) \, f(\xi,y) \; d\mu(\xi)
\end{equation}
where it might be abuse of notation. The rank-one kernel then has the representation
\begin{equation}
    (P F)(x,y) = \alpha u_0(x) \Phi_\xi(F(\xi,y)), \qquad P = \alpha u_0 \otimes \Phi_\xi
\end{equation}
For $n \geq 1$, define the iterated kernels $K^{(n)}$ by
\begin{equation} \label{it1}
    K^{(1)}(x,y):=K(x,y), \qquad K^{(n+1)}(x,y) := \int_X K(x,\xi) \, K^{(n)}(\xi,y) \; d\mu(\xi).
\end{equation}
Then, for every $n \geq 1$,
\begin{equation}
    (T^{n}f)(x) = \int_X K^{(n)}(x,y) \, f(y) \; d\mu(y).
\end{equation}

%%%%%
\subsection{Recursive definition of the corrected kernels} \label{ss:corrected}
We introduce the corrected kernels that remove the rank-one contribution at each iterate. To be more precise, the decomposition given in Proposition~\ref{prop:resolvent_decomposition} isolates the rank-one contribution of the resolvent at the operator level. To obtain a kernel-level expansion, it is therefore natural to remove this rank-one component at each iteration of the kernel. The corrected kernels introduced below implement this cancellation recursively.

\begin{Def}[Corrected kernel recursion]
    \rm Set
    \begin{equation}
        \Gamma_0 := K,
    \end{equation}
    and for $n \geq 1$ define recursively
    \begin{equation} \label{eq:gamma-corrected}
        \Gamma_n(x,y) := \int_X K(x,\xi) \, \Gamma_{n-1}(\xi, y) \; d\mu(\xi) - \alpha u_0(x) \, \Phi_\xi(\Gamma_{n-1}(\xi,y)).
    \end{equation}
\end{Def}

\noindent When we denote by
\begin{equation} \label{eq:corrected-kernel}
    R(x,y) := K(x,y) - \alpha u_0(x) g_0(y)
\end{equation}
we can simplify the recursion as
\begin{equation}
    \Gamma_n(x,y) = \int_X R(x,\xi) \, \Gamma_{n-1}(\xi,y) \; d\mu(\xi)
\end{equation}
which is the kernel-level expression of $\Gamma_n = R \, \Gamma_{n-1}$. Therefore we have the operator-level recursion:
\begin{equation} \label{eq:recursion}
    \Gamma_{n+1} = (T - P) \Gamma_{n}, \qquad n \geq 0.
\end{equation}
This recursion expresses the idea that $\Gamma_n$ is the kernel of $(T - P)^n \, K$, that is, the usual iterates of $T$, but corrected at each step by subtracting the rank-one component.

\begin{Thm} \label{thm:Bell}
    Let $T$ be a bounded positive operator possible to decompose into rank-one $P$ and the remainder $R$. For $\Gamma_n = R \, \Gamma_{n-1} = (T - P) \Gamma_n$, we have the following expression 
    \begin{equation}
        \Gamma_n = K^{(n)} - \sum_{\ell=0}^{n-1} (-1)^\ell \sum_{k=0}^{n-\ell-1} K^{(n-k-\ell-1)} \; B_{\ell+1,k+\ell+1}(b_1,b_2,\cdots).
    \end{equation}
    where $b_j := \Phi_\xi(K^{(j)}u_0)$ and $B_{p,q}(b_1,b_2,\cdots)$ are the partial Bell polynomials with parameter $p,q$.
\end{Thm}

\begin{Rem}
    \rm The simplest is the case that the integral kernel has the separable expression $K(x,y) = v(x) \, w(y)$. The expression~\eqref{eq:corrected-kernel} implies
    \begin{equation}
        K(x,y) = \alpha u_0(x) \, g_0(y) + R(x,y),
    \end{equation}
    which is the kernel-level decomposition comparing to the operator-level decomposition \eqref{eq:decomposition_T}.
\end{Rem}

%%%%%%
\subsection{Expansion of \texorpdfstring{$\Gamma_n$}{} and Bell polynomial structure} \label{ss:Bell}

In this section we show the proof of Theorem~\ref{thm:Bell}. Iterating \eqref{eq:recursion} gives
\begin{equation} \label{eq:iteration}
    \Gamma_n = (T - P)^n \, K.
\end{equation}
Expanding the product leads to alternating sums of the form $T^{i_1} \, P T^{i_2} \, P \cdots P \, T^{i_p} \, K$, with $p$ occurrences of $P$ and total exponent sum $i_1 + i_2 + \cdots + i_p = q$. Using the kernel representation
\begin{equation}
    P \, T^j \, K = w \, \Phi_\xi(T^j \, K),
\end{equation}
each block contributes a scalar term $b_j := \Phi_\xi(K^{(j)}u_0)$, and the sum over all placements of $p$ blocks of length $i_1, i_2, \cdots, i_p$ gives
\begin{equation}
    \sum_{i_1+i_2+\cdots+i_p=q} \; \prod_{j=1}^p b_{i_j}.
\end{equation}
Ignoring the ordered structure and counting only multiplicities of block lengths produces the partial Bell polynomial
\begin{equation} \label{eq:Bell_polynomial}
    B_{p,q}(b_1,b_2,\cdots) = \sum_{\substack{k_1+k_2+\cdots=p\\k_1+2k_2+\cdots=q}} \; \frac{p!}{k_1! \, k_2! \cdots} \; \prod_{j=1}^p b_j^{k_j}.
\end{equation}
Thus the kernel $\Gamma_n$ admits the representation
\begin{equation}
    \Gamma_n = K^{(n)} - \sum_{\ell=0}^{n-1} (-1)^\ell \sum_{k=0}^{n-\ell-1} K^{(n-k-\ell-1)} \; B_{\ell+1,k+\ell+1}(b_1,b_2,\cdots).
\end{equation}
Although the combinatorics is not essential for the resolvent expansion, the Bell-polynomial structure confirms that the recursive definition \eqref{eq:gamma-corrected} introduces no contradictions and yields the correct cancellation of rank-one contributions.

The iteration \eqref{eq:iteration} also yields the kernel-level decomposition. We recall here that the operator $R_\lambda$ acts here on kernels via the lifted action described in Section~\ref{ss:corrected}, so that $R_\lambda K$ is well-defined as a kernel.

\begin{Cor} \label{thm:fredholm-corrected}
    Let $K$ be Schur-bounded, and let $\Gamma_n$ be defined by \eqref{eq:gamma-corrected}. For $|\lambda|$ sufficiently large, the series
    \begin{equation} \label{eq:Neumann}
        (\lambda I - R)^{-1} K = \sum_{n=1}^\infty \lambda^{-(n+1)} \, \Gamma_n
    \end{equation}
    converges absolutely in operator norm satisfies the identity
    \begin{equation} \label{eq:fk-identity-corrected}
        (\lambda I - T) \, R_\lambda = K - P R_\lambda, \qquad P = \alpha u_0 \otimes \Phi_\xi.
    \end{equation}
    Equivalently,
    \begin{equation}
        R_\lambda = (\lambda I - T)^{-1} \, (K - P R_\lambda).
    \end{equation}
\end{Cor}

\begin{proof}
    Using $\Gamma_{n+1} = T \, \Gamma_{n} - P \, \Gamma_{n}$
    \begin{equation}
        (\lambda I - T) \, \lambda^{-n} \, \Gamma_n = \lambda^{-n} \, \Gamma_n - \lambda^{-n} \, T \Gamma_n
    \end{equation}
    But
    \begin{equation}
        T \Gamma_n = (T - P) \Gamma_n + P \Gamma_n = \Gamma_{n+1} + P \Gamma_n
    \end{equation}
    Thus
    \begin{equation}
        (\lambda I -T) \lambda^{-n} \, \Gamma_n = \lambda^{-n+1} \Gamma_n - \lambda^{-n} \, \Gamma_{n+1} + \lambda^{-n} \, P \Gamma_n.
    \end{equation}
    Summing over $n \geq 1$ gives a telescoping cancellation on the first two terms, leaving
    \begin{equation}
        (\lambda I - T) R_\lambda = \Gamma_1 - P R_\lambda = K - P R_\lambda.
    \end{equation}
    This proves \eqref{eq:fk-identity-corrected}.
\end{proof}

\noindent The term $K - P R_\lambda$ can be interpreted as a corrected kernel from which the rank-one feedback induced by the Doeblin minorization has been removed. This subtraction implements, at the kernel level, the decomposition of the resolvent and yields a renewal-type recursion for the corrected kernels $\Gamma_n$.

%%%%%%%%%%
\appendix

%%%%%%%%%%
\section{Change of Measure} \label{s:change_meas}

Let $(X,\mathcal F,\nu)$ be a $\sigma$-finite measure space and let 
\begin{equation}
    d\mu(x) = h(x) \, d\nu(x), \qquad h(x)>0.
\end{equation}
For a nonnegative kernel $K(x,y)$ define
\begin{equation}
    (T_\mu f)(x) = \int_X K(x,y) f(y)\, d\mu(y), \qquad f\in L^p(X,\mu).
\end{equation}

%%%%%%
\subsection{Isometric conjugation}

We fix conjugate exponents $p, q \in [1,\infty]$ with $1/p + 1/q = 1$, so that all integrals below are well defined. Define the map
\begin{equation}
    (Mf)(x) = h(x)^{1/p} f(x).
\end{equation}
Then $M \colon L^p(X,\mu)\to L^p(X,\nu)$ is an isometric isomorphism.  
A direct computation shows
\begin{equation}
    (M T_\mu M^{-1} g)(x) = \int_X K_e(x,y)\, g(y)\, d\nu(y),
\end{equation}
where
\begin{equation}
    K_e(x,y) = h(x)^{1/p} K(x,y) h(y)^{1/q}.
\end{equation}
Hence
\begin{equation}
    T_\nu := M T_\mu M^{-1}
\end{equation}
has the same operator norm, spectrum, and positivity properties as $T_\mu$. Thus, changing the underlying measure simply conjugates the operator by an isometry.

%%%%%
\subsection{Invariance of Schur bounds}

Assume the Schur estimates hold on $(X,\mu)$:
\begin{equation}
    \int_X K(x,y)\,\psi(y)\, d\mu(y) \leq C\, \phi(x), \qquad \int_X K(x,y)\,\phi(x)\, d\mu(x) \leq C\, \psi(y).
\end{equation}
Define transformed weights
\begin{equation}
    \tilde{\phi}(x) = h(x)^{-1/p} \phi(x), \qquad \tilde{\psi}(x) = h(x)^{1/q} \psi(x).
\end{equation}
Then $K_e$ satisfies the same inequalities with respect to $\nu$:
\begin{equation}
    \int_X K_e(x,y)\,\tilde{\psi}(y)\, d\nu(y) \leq C\, \tilde{\phi}(x), \qquad \int_X K_e(x,y)\,\tilde{\phi}(x)\, d\nu(x) \leq C\, \tilde{\psi}(y).
\end{equation}
Thus Schur-type boundedness is invariant under changes of measure.

\medskip

\begin{Rem}
\rm From the operator-theoretic viewpoint, all positive, spectral, and Fredholm–-Krein properties are preserved when passing from $T_\mu$ to $T_\nu$.
\end{Rem}

%%%%%%%%%%
\section{Perron--Frobenius as a Corollary} \label{s:PF}

In this section we illustrate how the abstract theory developed above recovers the classical Perron--Frobenius theorem for primitive nonnegative matrices as a direct corollary.

%%%%%
\subsection{Perron--Frobenius theorem under Doeblin minorization}

Let $A = (a_{ij})_{1 \leq i,j \leq n}$ be a nonnegative $n\times n$ matrix. We consider $A$ as an operator on $\mathbb{R}^n$ equipped with the standard order and norm. Throughout, we identify $\mathbb{R}^n$ with the Banach lattice $E = \ell^p(\{1, \dots, n\})$ for some $1 \leq p \leq\infty$ using counting
measure.

\begin{Assump}[Discrete Doeblin minorization] \label{assump:PF-Doeblin}
    \rm There exist $\alpha > 0$, a vector $u \in \mathbb{R}^n$ with $u_i > 0$ for all $i$, and a vector $g \in \mathbb{R}^n$ with $g_j \geq 0$ such that
    \begin{equation} \label{eq:PF-Doeblin}
        a_{ij} \; \geq \; \alpha\, u_i\, g_j \qquad \forall\, 1 \leq i,j \leq n.
    \end{equation}
    Moreover, there exists a strictly positive linear functional $\Phi \colon \mathbb{R}^n \to \mathbb{R}$, i.e. \, $\Phi(x) > 0$ for all $x \in \mathbb{R}^n_+ \setminus \{0\}$.
\end{Assump}

In the discrete setting, \eqref{eq:PF-Doeblin} is exactly the Doeblin-type rank-one minorization from Assumption~1.1 applied to the counting measure on $\{1,\dots,n\}$, with kernel $K(i,j) = a_{ij}$.

\begin{Cor}[Perron--Frobenius for primitive matrices]
 \label{cor:PF-primitive}
    Let $A$ be a nonnegative $n \times n$ matrix satisfying Assumption~\ref{assump:PF-Doeblin}. Then the following hold:
    \begin{enumerate}
        \item The spectral radius $\rho(A)>0$ is an eigenvalue of $A$. There exists a vector $w \in \mathbb{R}^n$ with $w_i > 0$ for all $i$ such that
        \begin{equation}
            A w = \rho(A)\, w.
        \end{equation}
        \item The eigenspace associated with $\rho(A)$ is one-dimensional, and $\rho(A)$ is the only eigenvalue of $A$ with modulus $|\lambda| = \rho(A)$.
        \item The resolvent $(\lambda I - A)^{-1}$ has a simple pole at $\lambda = \rho(A)$, and the residue is the rank-one projection onto the Perron eigenspace:
        \begin{equation}
            \text{Res}_{\lambda=\rho(A)} (\lambda I - A)^{-1} = w\otimes \Phi,
        \end{equation}
        where $\Phi$ is a strictly positive left eigenfunctional
        satisfying $\Phi A = \rho(A)\Phi$ and normalized by $\Phi[w]=1$.
    \end{enumerate}
\end{Cor}

\begin{proof}
    We view $A$ as the operator $T$ acting on $E = \ell^p (\{1,\dots,n\})$ with counting measure. Then Assumption~\ref{assump:PF-Doeblin} is exactly Assumption~\ref{assump:doeblin} in the present discrete setting, with kernel $K(i,j) = a_{ij}$, rank-one operator $P_0 = u \otimes \Phi$, and $T = A$. Since $E$ is finite-dimensional, quasi-compactness is automatic and the decomposition
    \begin{equation}
        T = \alpha P_0 + R
    \end{equation}
    holds with $R \geq 0$ as in Proposition~\ref{prop:quasi-compact}.

    \smallskip
    
    By Lemma~\ref{lem:positive-improving} (irreducibility under Doeblin minorization), $T$ is positive improving and irreducible. Hence Proposition~\ref{prop:resolvent_decomposition} applies and yields:
    \begin{itemize}
        \item the spectral radius $\rho(T)$ is a positive eigenvalue with a strictly positive eigenvector $w \in  E_+$;
        \item $\rho(T)$ is simple and the unique eigenvalue on the spectral circle $\{z \in \mathbb{C} \colon |z| = \rho(T)\}$;
        \item the corresponding spectral projection is $w \otimes \Phi$, where $\Phi$ is a strictly positive eigenfunctional normalized by $\Phi[w] = 1$.
    \end{itemize}

    \smallskip
    
    Since $T = A$, this is precisely the assertion of the corollary with $\rho(A) = \rho(T)$.
\end{proof}

%%%%%
\subsection{Power-Doeblin condition and Perron–Frobenius structure} \label{ss:power-doeblin}

In the previous sections we assumed a Doeblin-type rank-one minorization for the operator $T$ itself. In many discrete or Markovian settings, a natural weakening is to require a Doeblin condition for some power $T^N$. This is closely related to the classical Perron--Frobenius theory, where one works with powers of a nonnegative matrix.

\begin{Assump}[Power-Doeblin condition] \label{assump:power-Doeblin}
    \rm Let $T$ be a positive bounded operator on a Banach lattice $E$. We say that $T$ satisfies the \emph{power-Doeblin condition} if there exist $N \in \mathbb{N}$, $\alpha > 0$, $u_0 \in E_+$, and a strictly positive functional $\Phi \in E'_+$ such that
        \begin{equation} \label{eq:power-Doeblin}
            T^N f \; \geq \; \alpha \, u_0 \, \Phi[f] \qquad \text{for all } f\in E_+ \setminus\{0\}.
        \end{equation}
\end{Assump}

\noindent Thus $T^N$ satisfies the Doeblin-type minorization from Assumption~\ref{assump:doeblin}, and in particular $T^N$ is positive improving, irreducible and quasi-compact.

\begin{Cor}[Peripheral spectrum under the power-Doeblin condition] \label{cor:power-Doeblin-peripheral}
    Assume that $T$ is a positive bounded operator on $E$ and that Assumption~\ref{assump:power-Doeblin} holds. Then the following properties are satisfied:
    \begin{enumerate}
        \item The spectral radius $\rho(T) > 0$ is an eigenvalue of $T$.
        \item The operator $T^N$ has a unique dominant eigenvalue
        $\rho(T^N) = \rho(T)^N$, which is real, positive and algebraically simple. The corresponding eigenspace is one-dimensional and generated by a strictly positive vector $w\in E_+$.
        \item Every eigenvalue $\lambda$ of $T$ with modulus $|\lambda| = \rho(T)$ satisfies
        \begin{equation}
            \lambda^N \; = \; \rho(T)^N.
        \end{equation}
        In particular, the peripheral spectrum of $T$ is a finite subset of the set
        \begin{equation}
            \{\rho(T)\, \zeta \colon \zeta^N = 1\}.
        \end{equation}
        \item The spectral projection $P$ of $T^N$ associated with
        $\rho(T^N)$ is rank-one and can be written as
        \begin{equation}
            P = w \otimes \Phi_N,
        \end{equation}
        where $\Phi_N$ is a strictly positive eigenfunctional of $T^N$ normalized by $\Phi_N[w]=1$. Moreover, $P$ is also the sum of spectral projections of $T$ corresponding to the peripheral eigenvalues $\lambda$ with $|\lambda| = \rho(T)$ and $\lambda^N = \rho(T)^N$.
    \end{enumerate}
\end{Cor}

\begin{proof}
    By Assumption~\ref{assump:power-Doeblin}, the operator $T^N$ satisfies the Doeblin minorization \eqref{eq:power-Doeblin} with respect to the vector $u_0$ and the functional $\Phi$. Hence all results from Sections~\ref{s:structure}--\ref{s:spectral_dominance} apply to $T^N$ in place of $T$. In particular, the argument in Section~\ref{s:spectral_dominance} yields:
    \begin{itemize}
        \item $\rho(T^N)>0$ is the unique dominant eigenvalue of $T^N$;
        \item $\rho(T^N)$ is algebraically simple;
        \item the eigenspace of $\rho(T^N)$ is one-dimensional and generated by a strictly positive vector $w$;
        \item the corresponding spectral projection is rank-one and equals $w \otimes \Phi_N$ for some strictly positive eigenfunctional $\Phi_N$ with $\Phi_N[w] = 1$.
    \end{itemize}
    Since $\rho(T^N) = \rho(T)^N$ for a bounded operator, this proves (2) and (4). The existence of a positive eigenvalue $\rho(T)$ for $T$ itself follows by applying $T$ to $w$ and using the positivity of $T$ together with the one-dimensionality of the eigenspace of $T^N$.

    \smallskip
    
    If $\lambda$ is an eigenvalue of $T$ with eigenvector $v\neq 0$, then
    \begin{equation}
        T v = \lambda v \quad \Longrightarrow \quad T^N v = \lambda^N v.
    \end{equation}
    Thus $\lambda^N$ is an eigenvalue of $T^N$. If $|\lambda| = \rho(T)$, then $|\lambda^N| = \rho(T)^N = \rho(T^N)$, so by the spectral dominance of $T^N$ we must have $\lambda^N = \rho(T^N) = \rho(T)^N$. This proves (3).

    \smallskip
    
    Finally, since the eigenspace associated with $\rho(T^N)$ is one-dimensional and invariant under $T$, the restriction of $T$ to this eigenspace is diagonalizable with eigenvalues given by the peripheral eigenvalues of $T$.
\end{proof}

%%%%%
\subsection{An example: power-Doeblin without Doeblin for \texorpdfstring{$T$}{}}

We give a simple finite-dimensional example where the operator $T$ itself does not satisfy a Doeblin-type rank-one minorization, while a power $T^N$ does. We consider the $2 \times 2$ stochastic matrix
\begin{equation}
    P =
    \begin{pmatrix}
        0   & 1 \\
        \frac{1}{2} & \frac{1}{2}
    \end{pmatrix},
\end{equation}
viewed as a positive operator on $\mathbb{R}^2$ with the standard order.

This example shows that even when $P$ does not admit a rank-one Doeblin lower bound (because of zero entries), a power $P^2$ can satisfy such a bound once the chain has had enough time to ``mix'' across states. In the language of Section~\ref{ss:power-doeblin}.2, $P$ satisfies the power-Doeblin condition with $N = 2$. Applying the argument in Section~\ref{s:spectral_dominance} to $T = P^2$ yields:
\begin{itemize}
    \item a unique dominant eigenvalue $\rho(P^2) = \rho(P)^2$ which is real, positive and simple;
    \item a strictly positive eigenvector $w > 0$ for $P^2$;
    \item a rank-one spectral projection for $P^2$ given by $w \otimes \Phi_2$.
\end{itemize}
Since $P$ and $P^2$ share the same spectral radius, this information translates into detailed spectral information about $P$ itself, in the spirit of the Perron--Frobenius theorem for primitive matrices.

\begin{Rem}[Polynomial averaging vs. Perron--Frobenius]
\rm A polynomial $p(T)$ may satisfy a Doeblin-type minorization even when neither $T$ nor any power $T^k$ does. This averaged positivity collapses the peripheral spectrum of $T$ into a single dominant eigenvalue of $p(T)$. However, it does not restore the Perron--Frobenius simplicity for $T$ itself, nor the convergence of $T^n$. This observation highlights that power-Doeblin conditions are essentially optimal for recovering the strong Perron--Frobenius structure of $T$.
\end{Rem}

%%%%%%%%%%
\section{A kernel-space rank-one subtraction} \label{s:kernel-rankone}

This appendix explains how the point-evaluation subtraction recursion
\begin{equation} \label{eq:target-rec}
    \begin{aligned}
        &\Gamma_{0}(x,y) = K(x,y),\\
        &\Gamma_{n+1}(x,y) = \int_{\Omega} K(x,\xi) \Gamma_n(\xi,y) \, d\mu(\xi) - K(x,y) \, \Gamma_n(x_0,y_0), \quad n \geq 0,
    \end{aligned}
\end{equation}
arises naturally as a rank-one decomposition on a kernel space. This point of view clarifies why the subtraction term involves $K(x,y)$, and how it relates to the $L^p$-operator-theoretic framework used in the main text. As in \cite{CKO2025}, this recursion plays an important role to analyze integral projection models (IPM), especially to construct the solution of nonnegative Fredholm integral equation without Fredholm determinant. The recursion~\eqref{eq:target-rec} is not merely a heuristic calculation; it has a mathematical background involving the lift of operators from the Banach function space $E$ to the Bochner space. 

%%%%%
\subsection{Kernel space and the lifted operator} \label{ss:lifted}
Let $(\Omega,\mathcal{F},\mu)$ be a $\sigma$-finite measure space and let $E$ be a Banach function space on $(\Omega,\mu)$ (e.g. $E = L^p(\Omega)$). We consider the kernel space
\begin{equation} \label{eq:kernel-space}
    \mathcal{E} := L^\infty(\Omega;E) = \Bigl\{ F \colon \Omega \times \Omega \to \mathbb{C} \, \text{ measurable} \mid \, \|F\|_{\mathcal{E}} := \text{ess sup}_{y\in\Omega} \|F(\cdot,y)\|_{E} < \infty \Bigr\}.
\end{equation}

\begin{Rem}[No order-continuity]
    \rm Since $\mathcal{E}$ contains $L^\infty$, it cannot have order-continuous norm. Even in the case of intermittent maps breaking Doeblin condition, we can keep the rank-one structure as below.
\end{Rem}

\medskip

Let $K \colon \Omega \times \Omega \to [0,\infty)$ be a measurable kernel such that the integral operator
\begin{equation}
    (Tf)(x) := \int_\Omega K(x,\xi) \, f(\xi) \; d\mu(\xi)
\end{equation}
defines a bounded positive operator on $E$. Define the lifted operator $\mathcal{T} \colon \mathcal{E} \to \mathcal{E}$ by acting with $T$ in the first variable:
\begin{equation} \label{eq:app-lifted-op}
    (\mathcal T F)(x,y) := \int_\Omega K(x,\xi) \, F(\xi,y) \; d\mu(\xi), \qquad F \in \mathcal{E}.
\end{equation}
Since For a.e.~$y$, we have $(\mathcal{T} F)(\cdot,y) = T(F(\cdot,y))$ in $E$; hence $\|(\mathcal{T} F)(\cdot,y)\|_E \leq \|T\| \, \|F(\cdot,y)\|_E$ and taking $\text{ess sup}_y$ yields if $T \colon E \to E$ is bounded, then $\mathcal{T} \colon \mathcal{E} \to \mathcal{E}$ is bounded and $\|\mathcal{T}\|_{\mathcal{E}\to\mathcal{E}} \leq \|T\|_{E\to E}$. Moreover, if $T$ is positive, then $\mathcal{T}$ is also positive.

\begin{Rem}[Global vs local]
    \rm Consider the dynamics $\mathcal{T}^n$. If we have order-continuity, then we have $L^p$-evaluation globally. Since around the reference point $(x_0,y_0)$ the local time diverges, the compactness of the operator is lost. Thus we cannot have $L^p$-evaluation and control the singularity by uniform $L^\infty$-evaluation with respect to $y$. 
\end{Rem}

%%%%%
\subsection{The mollified kernel recursion}

Let $(\psi_\varepsilon)_{\varepsilon>0}$ and $(\eta_\delta)_{\delta>0}$ be standard approximate identities centered at $x_0$ and $y_0$, respectively:
\begin{equation}
    \psi_\varepsilon \geq 0, \quad \int\psi_\varepsilon = 1, \quad \text{supp}(\psi_\varepsilon) \subset B_\varepsilon(x_0), \qquad \eta_\delta \geq 0, \quad \int\eta_\delta = 1, \quad \text{supp}(\eta_\delta) \subset B_\delta(y_0).
\end{equation}
Define the continuous positive functional $\Phi_{\varepsilon,\delta} \colon \mathcal{E} \to \mathbb{R}$ by
\begin{equation} \label{eq:app-mollified-eval}
    \Phi_{\varepsilon,\delta}[F] := \int_\Omega \int_\Omega F(x,y) \, \psi_\varepsilon(x) \, \eta_\delta(y) \; d\mu(x) \, d\mu(y).
\end{equation}
Then $\Phi_{\varepsilon,\delta} \in (\mathcal{E})'_+$ whenever the map $y \mapsto \int |F(x,y)| \psi_\varepsilon(x) \, d\mu(x)$ is essentially bounded, which holds for instance when $E = L^p$ and $\psi_\varepsilon \in L^q$ with $1/p + 1/q = 1$. By H\"older's inequality, $\Phi_{\varepsilon,\delta}$ is continuous on $\mathcal{E}$: for $F \in \mathcal{E} = L^\infty(\Omega; L^p(\Omega))$,
\begin{equation}
    \Big| \Phi_{\varepsilon,\delta}[F] \Big| \leq \|\psi_\varepsilon\|_{L^{p'}} \, \|\eta_\delta\|_{L^1} \, \sup_{y\in\Omega} \|F(\cdot,y)\|_{L^p},
\end{equation}
and hence $\Phi_{\varepsilon,\delta} \in \mathcal E'_+$.
Next, view the kernel $K$ itself as an element of $\mathcal{E}$:
\begin{equation} \label{eq:app-K-in-kernel-space}
    K \in \mathcal{E}, \qquad \text{i.e.} \quad \text{ess sup}_{y} \|K(\cdot,y)\|_E < \infty,
\end{equation}
which is satisfied in typical $L^1$--$L^\infty$ Schur settings. Define the rank-one operator $\mathcal{P}_{\varepsilon,\delta} \colon \mathcal{E} \to \mathcal{E}$ by
\begin{equation} \label{eq:app-rankone-P}
    \mathcal{P}_{\varepsilon,\delta} := K \otimes \Phi_{\varepsilon,\delta}, \qquad (\mathcal{P}_{\varepsilon,\delta} F)(x,y) := K(x,y) \, \Phi_{\varepsilon,\delta}[F].
\end{equation}
Clearly $\mathcal{P}_{\varepsilon,\delta}$ is positive and
\begin{equation}
    R(\mathcal{P}_{\varepsilon,\delta}) = \text{span} \{K\} \subset \mathcal{E},
\end{equation}
hence $\mathcal{P}_{\varepsilon,\delta}$ is rank-one on $\mathcal{E}$. We remark here that since $\|\mathcal{P}_{\varepsilon,\delta}\| \leq \|K\|_\mathcal{E} \|\Phi_{\varepsilon,\delta}\| \leq \|K\|_\mathcal{E}$, $\|\mathcal{P}_{\varepsilon,\delta}\|$ is also bounded. Define the rank-one subtracted operator
\begin{equation} \label{eq:app-S-def}
    \mathcal{S}_{\varepsilon,\delta} := \mathcal T - \mathcal{P}_{\varepsilon,\delta}.
\end{equation}
Starting from $\Gamma_0 := K \in \mathcal{E}$, define
\begin{equation} \label{eq:app-Gamma-def}
    \Gamma_0^{\varepsilon,\delta} := K, \qquad \Gamma_{n+1}^{\varepsilon,\delta} := \mathcal S_{\varepsilon,\delta} \Gamma_n^{\varepsilon,\delta}.
\end{equation}
Therefore, for every $n \geq 0$ and a.e.~$(x,y) \in \Omega \times \Omega$, we have the kernel recursion:
\begin{equation} \label{eq:app-recursion}
    \Gamma_{n+1}^{\varepsilon,\delta}(x,y) = \int_\Omega K(x,\xi) \, \Gamma_n^{\varepsilon,\delta}(\xi,y) \; d\mu(\xi) - K(x,y) \, \Phi_{\varepsilon,\delta}[\Gamma_n^{\varepsilon,\delta}].
\end{equation}
To obtain the desired point-evaluation coefficient $\Gamma_n(x_0,y_0)$, we take $(\varepsilon,\delta) \downarrow (0,0)$ in \eqref{eq:app-recursion}. This requires a mild regularity assumption ensuring that mollified evaluation converges to point evaluation. Assuming that $\Gamma_n^{\varepsilon,\delta} \to \Gamma_n$ in a mode compatible with mollified evaluation (e.g. pointwise a.e. with uniform local bounds, or convergence in $L^1_{\text{loc}}$), we may pass to the limit in \eqref{eq:app-recursion} to obtain exactly \eqref{eq:target-rec}. In particular, the coefficient in the subtraction term becomes the scalar $\Gamma_n(x_0,y_0)$, and the direction remains $K(x,y)$.

\begin{Rem}[Convergence of $\Phi_{\varepsilon,\delta}$ to $\delta_{(x_0,y_0)}$]
    \rm Since
    \begin{equation}
        \int_\Omega \|F(\cdot,y)\|_E \, \eta_\delta(y) \; dy \leq \|F\|_\mathcal{E} \int_\Omega \eta_\delta(y) \; dy = \|F\|_\mathcal{E},
    \end{equation}
    the mollified functional $\Phi_{\varepsilon,\delta}$ is bounded. Therefore, $\Phi_{\varepsilon,\delta}$ is continuous functional and $\|\Phi_{\varepsilon,\delta}\| \leq 1$. Moreover, if $(x_0,y_0)$ is a Lebesgue point, then
    \begin{equation}
        \Phi_{\varepsilon,\delta}[F] \xrightarrow[\varepsilon, \delta\to0]{} F(x_0,y_0).
    \end{equation}
    The point evaluation $\delta_{(x_0,y_0)}$ is not bounded on general $L^p$, but by interpreting it as a weak convergence of our constructed $\Phi_{\varepsilon,\delta}$, we can discuss the validity of recursion even for singular kernels.
\end{Rem}

%%%%%%%%%%
\section{Numerical Results and Validation} \label{s:numerical}

To demonstrate the validity and robustness of the spectral construction method based on Bell polynomials, this chapter presents numerical simulations using a specific integral kernel. We focus on the logarithmic kernel, defined as:
\begin{equation}
    R(x,y) = \frac{1}{(x+y+2) [\log (x+y+2)]^p}
\end{equation}
For $p > 1$, the operator $R$ induced from the kernel $R(x,y)$ is in a trace class. The choice of the kernel for $p = 1$ is critical and particularly significant because, in the context of classical Fredholm theory, it often fails to be of trace class on infinite domains, making the traditional definition of the Fredholm determinant $D(\lambda)$ problematic. To be more precise, we introduce
\begin{equation}
    u_0(x) = \mathbbm{1}_{[0,1]}(x), \quad \Phi(y) = \mathbbm{1}_{[0,1]}(y), \quad R(x,y) = \frac{1}{(x+y+2) \log (x+y+2)}
\end{equation}
and consider the space $L^1(0,\infty)$. We set $P = u_0 \otimes \Phi$. Thus, since
\begin{equation}
    \Phi[u_0] = \int_0^1 \mathbbm{1}_{[0,1]}(x) \; dx = 1,
\end{equation}
$P^2 = P$, i.e., $P$ is a projection. To define the determinant $\det (I - \frac{R}{\lambda})$, we need the compactness or trace-class property for $R$. However, it is not. Indeed, the trace of $R$ is
\begin{equation}
    \text{Tr}(R) = \int_0^\infty R(x,x) \; dx = \int_0^\infty \frac{1}{2(x+1) \log (2x+2)} = \int_{\log 2}^\infty \frac{1}{2u} \; du = \infty
\end{equation}
We cannot start to compute the characteristic equation. Moreover, $R$ is not compact, i.e., the spectral may be obscured by the essential spectrum (continuum).

%%%%%
\subsection{Renewal structures}
The example shows that the operator $T = P + R$ is positive and satisfies the Doeblin minorization condition. Through our argument, the recursion $\Gamma_{n+1} = R \, \Gamma_n, \; \Gamma_0 = u_0$ yields two renewals: $\Gamma_n$ itself shows the vector-valued renewal of status and $b_n = \Phi[\Gamma_n]$ gives the scalar-valued renewal of observations. We can compute each $\Gamma_1$ and $b_1$ as
\begin{equation}
    \Gamma_1 = R \, u_0 = \int_0^1 \frac{1}{(x+y+2) \log (x+y+2)} \; dx
\end{equation}
and
\begin{equation}
    b_1 = \Phi[R \, u_0] = \int_0^1 \int_0^1 \frac{1}{(x+y+2) \log (x+y+2)} \; dxdy
\end{equation}
Since the denominator is continuous function on $[0,1]\times[0,1]$, both are finite positive values. Thus $\Gamma_n$ and $b_n$ are also finite positive. Then we can define the scalar function $D(\lambda)$ and $w(\lambda)$. Since the (dominant) eigenvalue $\lambda_0$ can be characterized as the solution of $D(\lambda_0) = 0$, we can show the expression for eigenfunction by $w(\lambda_0)$. We can also obtain the explicit relation between $\Gamma_n$ and $b_n$
\begin{equation} \label{eq:bell}
    \Gamma_n = K^{(n)} - \sum_{\ell=0}^{n-1} (-1)^\ell \sum_{k=0}^{n-\ell-1} K^{(n-k-\ell-1)} \; B_{\ell+1,k+\ell+1}(b_1,b_2,\cdots),
\end{equation}
where $B_{p,q}(b_1,b_2,\cdots)$ are the partial Bell polynomials with parameter $p,q$. This shows the feedback from the observations $b_n$ to the status $\Gamma_n$, which forms the algebraic structure of eigen-system.

The dominant eigenvalue can be also characterized by the limit of ratio
\begin{equation}
    \lim_{n\to\infty} \frac{b_{n+1}}{b_n} = \lim_{n\to\infty} \frac{\Phi[R^{n+1} \, u_0]}{\Phi[R^n \, u_0]}
\end{equation}
Indeed, if we assume that $\lambda_0$ is the simple, isolated, dominant eigenvalue, $w$, $\psi$ are the right and left eigenfunctions (Perron--Frobenius or Krein--Rutman theorems can guarantee the existence and uniqueness), then, by spectral decomposition, we have
\begin{equation}
    R = \lambda_0 \underbrace{\frac{w \otimes \psi}{\langle \psi, w \rangle}}_{P_0} + R_{\text{rem}},
\end{equation}
where $\rho(R_{\text{rem}}) < \lambda_0$, $R_{\text{rem}} w = 0$ and $\psi R_{\text{rem}} = 0$.
\begin{equation}
    \Gamma_n = (\lambda_0^n P_0 + R_{\text{rem}}^n) \, u_0 = \underbrace{\frac{\langle \psi, u_0 \rangle}{\langle \psi, w \rangle}}_{c} \lambda_0^n w + R_{\text{rem}}^n u_0 = \lambda_0^n \bigl( c \, w + \varepsilon_n \bigr)
\end{equation}
Therefore,
\begin{equation}
    \begin{aligned}
        \frac{b_{n+1}}{b_n} = \frac{\lambda_0^{n+1} \bigl( c \, \Phi[w] + \Phi[\varepsilon_{n+1}]\bigr)}{\lambda_0^n \bigl( c \, \Phi[w] + \Phi[\varepsilon_n]\bigr)} = \lambda_0 \, \dfrac{1 + \frac{\Phi[\varepsilon_{n+1}]}{c \, \Phi[w]}}{1 + \frac{\Phi[\varepsilon_n]}{c \, \Phi[w]}} \xrightarrow[n\to\infty]{} \lambda_0
    \end{aligned}
\end{equation}
The limit of ratio provides the explicit computation formula of the eigenvalue. We have two contrasting formulae for the eigenvalue:
\begin{equation}
    \begin{aligned}
        &\text{characteristic equation (analytic, static)} \quad D(\lambda_0) = 0,\\
        &\text{algorithm (algebraic, dynamic)} \qquad \lim_{n\to\infty} \frac{b_{n+1}}{b_n} = \lambda_0.
    \end{aligned}
\end{equation}

\begin{Rem}
    \rm For $T = P + R$, the Fredholm theory considers the sum of all eigenvalues for $R$ but our approach induces dynamics through $P$, which is an interesting contrast between global and local. For the existing theory, it takes advantage of proving the existence of the eigenvalue. However, it fails to provide an explicit representation of the eigenfunction as a series. On the contrary, our approach provides the method to construct the eigenfunction through connecting observation $b_n$ and evolution $\Gamma_n$ by \eqref{eq:bell}.
\end{Rem}

%%%%%
\subsection{Numerical results}
By discretizing the operator into an $N \times N$ matrix and observing the evolution of the coefficients $b_n$, we provide empirical evidence that our local, constructive approach circumvents the global convergence issues typically associated with non-compact operators. The following sections detail the convergence characteristics and the resulting spectral profiles.

\begin{figure}[!ht]
  \centering
  \includegraphics[width=0.8\linewidth]{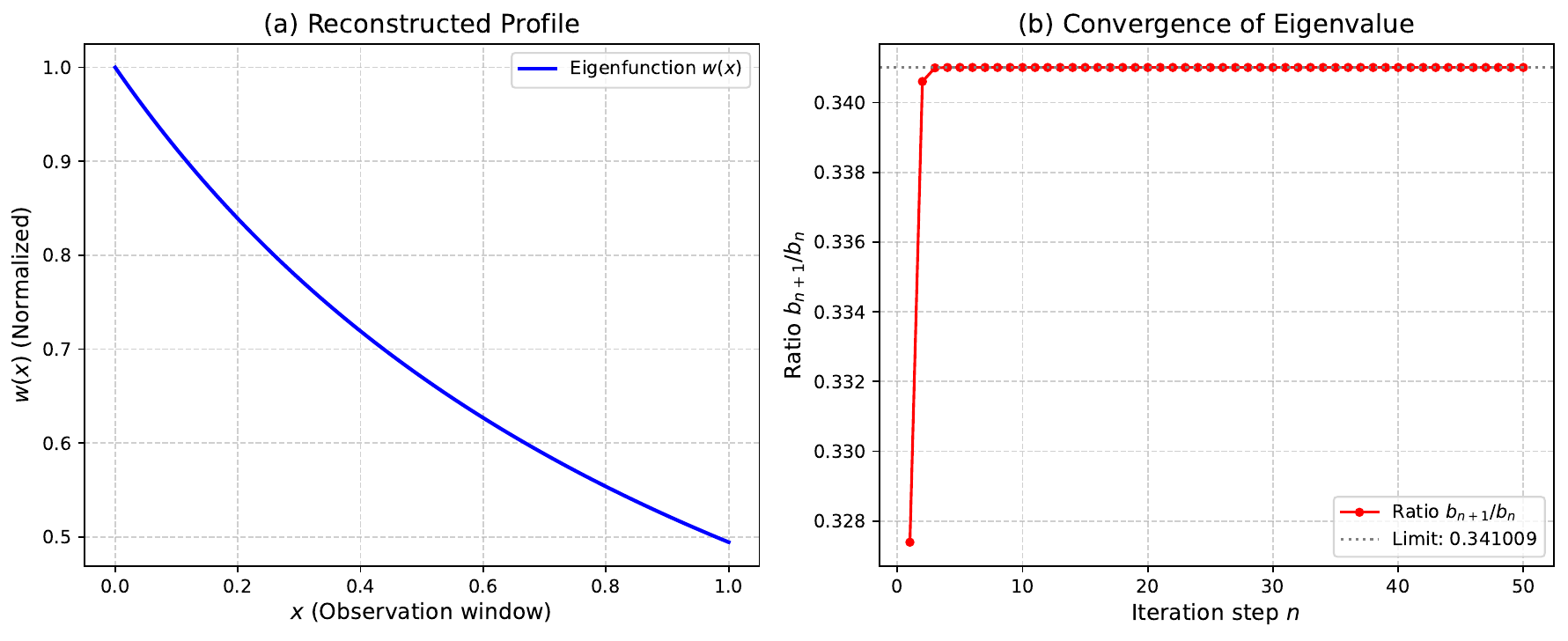}
  \caption{(Left) The dominant eigenfunction $w(x)$ obtained through the proposed method. Despite the non-trace-class nature of the operator, the iterative process converges to a stable, smooth, and positive profile within the observation window $\Phi$. (Right) The convergence of the ratio $\frac{b_{n+1}}{b_n}$}
  \label{fig:eigenfunction}
\end{figure}
\if0
\lstset{
    language=Python,
    basicstyle=\ttfamily\small,
    keywordstyle=\color{blue},
    commentstyle=\color{gray},
    breaklines=true,
    frame=single
}

\begin{lstlisting}[caption=Python implementation for eigenmode reconstruction]
import numpy as np
import matplotlib.pyplot as plt

N = 100
x = np.linspace(0, 1, N)
dx = x[1] - x[0]

M = np.array([[(1.0 / ( (xi + xj + 2) * np.log(xi + xj + 2) )) * dx for xj in x] for xi in x])

u = np.ones(N)
ratios = []

for n in range(50):
    u_next = M @ u
    ratios.append(np.sum(u_next) / np.sum(u))
    u = u_next

u_final = u / np.max(u)
lambda_0 = ratios[-1]

fig, (ax1, ax2) = plt.subplots(1, 2, figsize=(12, 5))

ax1.plot(x, u_final, label='Eigenfunction $w(x)$', color='blue', linewidth=2)
ax1.set_xlabel('$x$ (Observation window)', fontsize=12)
ax1.set_ylabel('$w(x)$ (Normalized)', fontsize=12)
ax1.set_title('(a) Reconstructed Profile', fontsize=14)
ax1.grid(True, linestyle='--', alpha=0.6)
ax1.legend(loc='best')

ax2.plot(range(1, len(ratios)+1), ratios, 'o-', color='red', markersize=4, label='Ratio $b_{n+1}/b_n$')
ax2.axhline(y=lambda_0, color='gray', linestyle=':', label=f'Limit: {lambda_0:.6f}')
ax2.set_xlabel('Iteration step $n$', fontsize=12)
ax2.set_ylabel('Ratio $b_{n+1}/b_n$', fontsize=12)
ax2.set_title('(b) Convergence of Eigenvalue', fontsize=14)
ax2.grid(True, linestyle='--', alpha=0.6)
ax2.legend(loc='best')

plt.tight_layout()

plt.savefig("combined_results.pdf", bbox_inches='tight')
plt.show()
\end{lstlisting}
\fi
%%%%%%%%%%
\section*{Acknowledgment}
The authors would like to dedicate this work to the late Professor Nobuhiko Fujii, and to express our heartfelt gratitude for the many fruitful discussion. The authors also thank to Hisashi Inaba for insightful discussions and valuable advice.

%\bibliographystyle{cas-model2-names}

%% bibliography
\printbibliography
%%%%%%%%%%
\end{document}